\documentclass[11pt]{article}
\usepackage{latexsym}
\usepackage{amsfonts}
\usepackage{mathrsfs,amsthm,amsmath}
\usepackage{todonotes}
\usepackage{color}
\newtheorem{theorem}{Theorem}[section]
\newtheorem{proposition}[theorem]{Proposition}

\newtheorem{lemma}{Lemma}[section]

\newtheorem{example}{Example}[section]

\begin{document}
\title{A large deviation principle for a class of weighted means of random variables which converges weakly
	to the Dickman distribution}
\author{Rita Giuliano\thanks{Address: Dipartimento di
		Matematica "L. Tonelli", Università di Pisa, Largo Bruno
		Pontecorvo 5, I-56127 Pisa, Italy. e-mail:
		\texttt{giuliano@dm.unipi.it}} \and Claudio Macci\thanks{Address:
		Dipartimento di Matematica, Università di Roma Tor Vergata, Via
		della Ricerca Scientifica, I-00133 Rome, Italy. e-mail:
		\texttt{macci@mat.uniroma2.it}}}
\date{}
\maketitle
\maketitle
\begin{abstract}\noindent
	In this paper we consider a wide class of weighted means of random variables which converge
	weakly to the Dickman distribution. This is inspired by a result proved in \cite{BhattacharjeeMolchanov}.
	Then we prove a large deviation principle for the sequence of these weighted means. In  particular 
	we recover a result proved in \cite{GiulianoMacciESAIM-PS}. Moreover the generalized framework  of  
	this paper allows to consider suitable Neyman Type A distributed random variables.\\
	
\end{abstract}

\section{Introduction}
 
The Dickman function $\rho$, introduced in \cite{Dickman}, is defined as the continuous solution 
of the delay differential equation
$$u\rho^\prime(u)+\rho(u-1)=0\ (u>1)$$
with the initial condition $\rho(u)=0$ for $u\in[0,1]$ (see, e.g., \cite{Tenenbaum}, Section III.5.4).
This function plays an important role in analytic number theory (see, e.g., \cite{DDD}; see also \cite{G}
for the connection with the \emph{Smooth Numbers}) and in other several fields (a list of references can be found 
in \cite{HwangTsai}). It is known (see \cite{Tenenbaum}) that the Laplace transform of $\rho$ is 
\begin{equation} \nonumber
	  \int_0^\infty e^{s x}\rho(x)dx=\exp\left( \gamma +\int_0^1\frac{e^{sy}-1}{y}dy\right)\
	\mbox{for all}\ s\in\mathbb{R},
\end{equation} 
where $\gamma$ is Euler's constant; thus $e^{-\gamma}\rho$ is a probability density, and the distribution with 
this density is the Dickman distribution $D_1$. More in general, for $\nu>0$, a random variable $Z$ has Dickman 
distribution $D_\nu$ if
\begin{equation}\label{eq:dickman-mgf}
	\mathbb{E}\left[e^{sZ}\right]= \exp\left(\nu\int_0^1\frac{e^{sy}-1}{y}dy\right)\
	\mbox{for all}\ s\in\mathbb{R}.
\end{equation}  
The moment generating function \eqref{eq:dickman-mgf} can also be found in other references: see
e.g. eq. (2.5) in \cite{Hensley} (where $K(\alpha)$ is identically equal to 1) and eq. (66) in 
\cite{AT}.  
  
\smallskip
\noindent
Some results on the Dickman distribution concerns weak convergence of some sequences of random variables; see for
instance \cite{AT}, \cite{ABT}, \cite{V} (Theorem 4.7.7), and \cite{GiulianoMacciESAIM-PS} 
(Proposition 2.5 and Section 4). More recent results concerning a generalized Dickman distribution can be found in \cite{P}.
Such results have interest in the study of  random combinatorial structures (see again  \cite{AT}, \cite{ABT}, \cite{P})
and number theory applications (see \cite{CS} and Section 4 in \cite{GiulianoMacciESAIM-PS}). Other asymptotic results
concern large deviations; see e.g. \cite{GiulianoMacciESAIM-PS} (Proposition 3.2).
The theory of large deviations gives an asymptotic computation of small probabilities on an exponential scale (see e.g.
\cite{DemboZeitouni} as a reference on this topic); some preliminiaries will be recalled in Section \ref{Preliminaries}. 

 \smallskip
 \noindent
The present paper has been stimulated by a more recent weak convergence result to the Dickman distribution presented in
\cite{BhattacharjeeMolchanov} (Theorem 3.3), which has some analogies with the weak convergence result in
\cite{GiulianoMacciESAIM-PS} cited above; here we prove a large deviation principle, analogous to the one proved in 
\cite{GiulianoMacciESAIM-PS}. The approach in \cite{BhattacharjeeMolchanov} seems to be more general than the one in 
\cite{GiulianoMacciESAIM-PS}, but we can obtain a large deviation principle with the same speed and the same rate function 
as in \cite{GiulianoMacciESAIM-PS} (see Theorem \ref{main} for a precise statement); in particular we 
can recover a particular case of the large deviation principle in \cite{GiulianoMacciESAIM-PS}.
However the more general approach suggested by the paper \cite{BhattacharjeeMolchanov} allows to consider other interesting 
cases; one of them concerns random variables with \emph{Neyman Type A} distributions (which are particular compound Poisson 
distributions, with nonnegative and integer valued random summands). To the best of our knowledge, the example 
presented in this paper involving Neyman Type A distributed random variables does not appear in the literature.

\smallskip
\noindent
We conclude with the outline of the paper. 
Section \ref{Preliminaries} is devoted to recall some preliminaries on large deviations and some results and examples from
\cite{GiulianoMacciESAIM-PS}. In Section \ref{framework} we present the generalized framework studied in this paper, i.e. 
Example \ref{ex:Dickman-more-general}; moreover we present the statement of Theorem \ref{main}, i.e. the large deviation 
principle for the sequence of weighted means described in Example \ref{ex:Dickman-more-general}. Section \ref{sec:discussion} 
is devoted to some discussion on Example \ref{ex:Dickman-more-general}: in particular we illustrate that a particular case of 
Example \ref{ex:Dickman-Poisson} can be recovered also as a particular case of Example \ref{ex:Dickman-more-general}, and 
moreover we present Example \ref{Poisson-Poisson} based on suitable Neyman Type A distributed random variables. We conclude 
with Section \ref{sec:final} in which we present some technical results (Section \ref{lemmas}) and the proof of Theorem 
\ref{main} (Section \ref{proof-of-the LDP}).

\section{Preliminaries}\label{Preliminaries}
We start with some preliminaries on large deviations. Moreover, we introduce the notations and 
recall some results of \cite{GiulianoMacciESAIM-PS} that form the basis of the results in this paper.

\subsection{Preliminaries on large deviations}
 We refer to \cite{DemboZeitouni} (pages 4-5). Let $\mathcal{Z}$ be
	a topological space equipped with its completed Borel
	$\sigma$-field. A sequence of $\mathcal{Z}$-valued random
	variables $\{Z_n:n\geq 1\}$ satisfies the large deviation
	principle (LDP for short) with speed function $v_n$ and rate
	function $I$ if: $\lim_{n\to\infty}v_n=\infty$; the function
	$I:\mathcal{Z}\to[0,\infty]$ is lower semi-continuous;
	$$\limsup_{n\to\infty}\frac{1}{v_n}\log P(Z_n\in F)\leq-\inf_{z\in
		F}I(z)\ \mbox{for all closed sets}\ F;$$
	$$\liminf_{n\to\infty}\frac{1}{v_n}\log P(Z_n\in G)\geq-\inf_{z\in
		G}I(z)\ \mbox{for all open sets}\ G.$$
	A rate function $I$ is said to be  {\it good} if its level sets
	$\{\{z\in\mathcal{Z}:I(z)\leq\eta\}:\eta\geq 0\}$ are compact.
	
	\noindent
	Throughout this paper we take $\mathcal{Z}=\mathbb{R}$ and we use the G\"artner--Ellis Theorem 
	(see e.g. Theorem 2.3.6 in \cite{DemboZeitouni}); the application of this theorem consists in checking the existence 
	of the function $\Lambda:\mathbb{R}\to(-\infty,\infty]$ defined by
    $$\Lambda(\theta):=\lim_{n \to \infty} \frac{1}{v_n}\log E[e^{v_n \theta Z_n}]$$
    and, if $\Lambda$ is essentially smooth (see e.g. Definition 2.3.5 in \cite{DemboZeitouni}) and lower semi-continuous, 
    the LDP holds with good rate function  $\Lambda^*: \mathbb{R}\to [0, \infty]$  defined by
  $$\Lambda^*(x) :=\sup_{\theta \in \mathbb{R}}\{\theta x - \Lambda (\theta)\}.$$

\subsection{Notations and brief recall of some results from \cite{GiulianoMacciESAIM-PS}}\label{gen-framework}
As explained in the Introduction, our aim is to consider a class of sequences of random variables $\{Z_n:n\geq 1\}$, 
where each random variable $Z_n$ is a    weighted partial sum of other random variables $ \{W_n  :n\geq 1\} $ (which 
converge weakly  to a Dickman distribution). Thus, first of all, we give the precise definition 
of $Z_n$. In particular throughout this section we often refer to the setting described in terms of the formulas \eqref{eq:conditions-on-phi(n)}--\eqref{eq:pmf-Rn} below.

\paragraph{Setting.} Let $\phi:\mathbb{N}\to [0,\infty)$ be a strictly increasing function such that
\begin{equation}\label{eq:conditions-on-phi(n)}
	\phi(0)=0,\ \lim_{n\to\infty}\phi(n)=\infty,\ \mbox{and}\
	\lim_{n\to\infty}\frac{\phi(n)}{\phi(n+1)}=1,
\end{equation}
and is concave, i.e.
\begin{equation}\label{eq:concave-phi}
	\{\phi(n)-\phi(n-1):n\geq 1\}\ \mbox{is non-increasing}.
\end{equation}
Moreover we define
\begin{equation}\label{eq:def-L(n)}
	L(n):=\sum_{k=1}^n\frac{\phi(k)-\phi(k-1)}{\phi(k)},
\end{equation}
and we assume that
\begin{equation}\label{eq:L(n)-divergent}
	\lim_{n\to\infty}L(n)=\infty.
\end{equation}
Then $\{Z_n:n\geq 1\}$ is the sequence defined as
\begin{equation}\label{eq:def-weighted-mean}
	Z_n:=\frac{1}{L(n)}\sum_{k=1}^n\frac{\phi(k)-\phi(k-1)}{\phi(k)}W_k,
\end{equation}
where $\{W_n:n\geq 1\}$ is a sequence of real valued random variables, and
\begin{equation}\label{eq:def-Wn}
	W_n:=\frac{1}{\phi(n)}\sum_{k=1}^n\phi(k)R_k
\end{equation}
for suitable independent, nonnegative and integer-valued random variables $\{R_n:n\geq 1\}$. In 
particular we use the notation
\begin{equation}\label{eq:pmf-Rn}
	q_n^{(j)}:=P(R_n=j)\quad\mbox{for all}\ j\geq 0.
\end{equation}

\noindent
Throughout this paper we consider the above setting, with some particular choices of the probability mass functions 
in \eqref{eq:pmf-Rn}. A case in which the random variables $\{R_n: n\geq 1\}$ are Poisson distributed and their 
means tend to 0 has been studied in \cite{GiulianoMacciESAIM-PS}; see the following Example \ref{ex:Dickman-Poisson} 
for a rigorous presentation.
 
\begin{example}\label{ex:Dickman-Poisson}
	We consider the above setting and, for every $n\geq 1$, we assume that the random variable $R_n$ is Poisson
	distributed with mean $\lambda_n$. We also assume that the sequence of positive numbers $\{\lambda_n:n\geq 1\}$
	satisfies the following condition:
	\begin{equation}\label{eq:condition-for-limit-law}
		\lim_{n\to\infty}\frac{1}{\phi(n)}\sum_{k=1}^n\phi(k)\lambda_k=\nu,\quad
		\mbox{for some}\ \nu\in(0,\infty).
	\end{equation}
    Note that \eqref{eq:conditions-on-phi(n)} and \eqref{eq:condition-for-limit-law} yield that
    $\lambda_n\to 0$.
\end{example}     
 
\noindent
Then the following results have been proved in \cite{GiulianoMacciESAIM-PS}.

\begin{proposition}[Proposition 2.5 in \cite{GiulianoMacciESAIM-PS}]\label{prop:dickman-poisson-weak-convergence}
	Let $\phi:\mathbb{N}\to [0,\infty)$ be a strictly increasing function such that 
	\eqref{eq:conditions-on-phi(n)} and \eqref{eq:concave-phi} hold. Moreover let $\{L(n):n\geq 1\}$
	be defined by \eqref{eq:def-L(n)}, and assume that \eqref{eq:L(n)-divergent} holds. 
	Let $\{\lambda_n:n\geq 1\}$ and $\{R_n:n\geq 1\}$ be  as in Example \ref{ex:Dickman-Poisson} (in particular recall 
	\eqref{eq:pmf-Rn}).  Finally, let $\nu$ be the constant defined by \eqref{eq:condition-for-limit-law}.
	Then $\{W_n:n\geq 1\}$ defined by \eqref{eq:def-Wn} converges weakly to $D_\nu$ (as $n\to\infty$).
\end{proposition}

\begin{proposition}[Proposition 3.2 in \cite{GiulianoMacciESAIM-PS}]\label{prop:dickman-poisson-LDP}
	Let $\phi:\mathbb{N}\to [0,\infty)$ be a strictly increasing function such that 
	\eqref{eq:conditions-on-phi(n)} and \eqref{eq:concave-phi} hold. Moreover let $\{L(n):n\geq 1\}$
	be defined by \eqref{eq:def-L(n)}, and assume that \eqref{eq:L(n)-divergent} holds. Finally let 
	$\{\lambda_n:n\geq 1\}$ and $\{R_n:n\geq 1\}$ be as in Example \ref{ex:Dickman-Poisson}. Then 
	$\{Z_n:n\geq 1\}$ defined by \eqref{eq:def-weighted-mean} and \eqref{eq:def-Wn} satisfies the 
	LDP with speed function $v_n=L(n)$ and good rate function $I$ defined by
	\begin{equation}\label{eq:dickman-rate-function}
		I(x):=\left\{\begin{array}{ll}
			x\log\frac{x}{\nu}-x+\nu&\ \mbox{if}\ x\geq 0\\
			\infty&\ \mbox{if}\ x<0,
		\end{array}\right.
	\end{equation}
	where $0\log 0=0$.
\end{proposition}

\section{A generalized framework and the statement of Theorem \ref{main}} \label{framework}
The aim of this paper is to consider a generalized framework for the setting above (Section \ref{gen-framework}).
In particular we consider a wide class of distributions for the random variables $\{R_n:n\geq 1\}$ by 
referring to their probability mass functions (see \eqref{eq:pmf-Rn}) and their moment generating functions.
A precise description will be given in Example \ref{ex:Dickman-more-general} in which all the items are
expressed in terms of the function $\phi$ and a sequence $\{\pi_n:n\geq 1\}$ that tends to 1 (as $n\to\infty$).
This class of distributions encompasses a special case of Example \ref{ex:Dickman-Poisson} studied in 
\cite{GiulianoMacciESAIM-PS}, and allows other cases; among them there is a particular case with
\emph{Neyman Type A} distributed random variables $\{R_n:n\geq 1\}$ (the Neyman Type A distribution is a particular
compound Poisson distribution) which will be studied in detail.

\smallskip
\noindent
Our interest in the generalized framework of this paper is motivated by the following weak convergence result 
in the literature, which can be formulated in terms of the setting in Section \ref{gen-framework}; in particular
the authors prove a weak convergence to a  Dickman distribution as in \cite{GiulianoMacciESAIM-PS} (see Proposition \ref{prop:dickman-poisson-weak-convergence}).

\begin{theorem}[Theorem 3.3 in \cite{BhattacharjeeMolchanov}]\label{BM}
   For a monotone sequence of positive numbers $\{\phi(n):n\geq 0\}$ increasing to infinity with 
   $\lim_{n \to \infty}   \frac{\phi(n-1)}{\phi(n)}  = 1$, let $\{R_n:n\geq 1\}$ be independent 
   and nonnegative integer-valued random variables, and, recalling  \eqref{eq:pmf-Rn}, assume that
   $$q_n^{(0)} =  \left(\frac{\phi(n-1)}{\phi(n)}\right)^\nu \qquad \hbox{and} \qquad  q_n^{(1)}=q_n^{(0)}(1-q_n^{(0)})$$
   for some $\nu > 0$. Assume in addition that 
   $ \mathbb{E}[R_n]=\mathcal{O}( q_n^{(1)} )$. Then the sequence $\{W_n:n\geq 1\}$ in \eqref{eq:def-Wn} converges weakly 
   to $D_\nu$ (as $n\to\infty$). 
\end{theorem} 

\smallskip
\noindent
Now we are ready to present the generalized framework studied in this paper.
  
\begin{example}\label{ex:Dickman-more-general}
	We consider the setting above (Section \ref{gen-framework}). Moreover let $\{\pi_n:n\geq 1\}$ be the sequence
    defined by
	\begin{equation}\label{eq:def-pi(k)}
		\pi_k:=\left(\frac{\phi(k-1)}{\phi(k)}\right)^\nu\quad(\mbox{for all}\ k\geq 1)\quad
		\mbox{for some}\ \nu\in(0,\infty).
	\end{equation}
    Moreover, we assume that:
	\begin{itemize}
		  \item [(a)]  for a sequence of functions $f_n:[0,1]\to[0,1]$ such that, uniformly in $n$
		\begin{equation}\label{eq:condition-function-f}
			\lim_{x\uparrow 1}f_n(x)=f_n(1)=1\quad \mbox{and}\quad \lim_{x\uparrow 1}\frac{f_n(x)-1}{x-1}=1,
		\end{equation}
	    we have
		\begin{equation}\label{eq:pmf-condition-0-and-1}
			q_k^{(0)} =f_k(\pi_k)\quad\mbox{and}\quad q_k^{(1)}=f_k(\pi_k)(1-\pi_k);
		\end{equation}
        \item[(b)] $\mathbb{E}[e^{\theta R_n}]=\sum_{h\geq 0}e^{\theta h}q_n^{(h)}<\infty$ for every $\theta>0$ and for every $n\geq 1$;
        \item[(c)] for every $\theta>0$, there exists $C_\theta>0$ and an integer $k_\theta\geq 1$ such that 
            \begin{equation}\label{eq:UB-quadratic}
        	0\leq\sup_{0<|\tau|\leq \theta }\Big\{\mathbb{E}[e^{\tau R_k}]-(q_k^{(0)}+e^\tau q_k^{(1)}) \Big\}\leq 
        	C_\theta(1-q_k^{(0)})^2\quad\mbox{for every}\ k>k_\theta.
        \end{equation}
   \end{itemize}
\end{example}

To motivate Example \ref{ex:Dickman-more-general}, we note that it provides a generalization of the function $f$ from Example \ref{ex:Dickman-Poisson} (see \eqref{eq:Dickman-more-general-conditions} below). Consequently, it allows us to consider new cases, such as the one involving \emph{Neyman Type A} distributed random variables introduced in Section \ref{poissonpoisson}. Moreover, Condition (c) in Example \ref{ex:Dickman-more-general} indicates that we are studying scenarios where, in a sense, the moment generating function $\mathbb{E}[e^{\tau R_k}]$ can be approximated by the sum of its first two terms, i.e., $q_k^{(0)}+e^\tau q_k^{(1)}$.

We also remark that, as already pointed out, for the sequence defined by \eqref{eq:def-pi(k)} 
we have $\pi_n\to 1$ (as $n\to\infty$) by \eqref{eq:conditions-on-phi(n)}. Then, for a future reference, we notice 
that Condition (a) in Example \ref{ex:Dickman-more-general} implies that
\begin{equation}\label{limiti}
	\lim_{n \to \infty}f_n(\pi_n)=1 \qquad {\it and} \qquad \lim_{n \to \infty}\frac{f_n(\pi_n)-1}{\pi_n-1}=1.
\end{equation}
This can be verified in a standard manner; the details are omitted.

We conclude this section with the statement of the main result in this paper, which is the analogous of Proposition 
3.2 in \cite{GiulianoMacciESAIM-PS} for Example \ref{ex:Dickman-more-general}.

\begin{theorem}\label{main}
	Consider the situation in Example \ref{ex:Dickman-more-general}.
	Then $\{Z_n:n\geq 1\}$  defined by \eqref{eq:def-weighted-mean} and \eqref{eq:def-Wn} satisfies the LDP 
	with speed function $v_n =L(n)$ and good rate function $I$ defined by 
	\eqref{eq:dickman-rate-function}.
\end{theorem}

\noindent Theorem \ref{main} will be proved in Section \ref{sec:final}.

\section{Some discussion on Example \ref{ex:Dickman-more-general}}\label{sec:discussion}
In this section we discuss some aspects of Example \ref{ex:Dickman-more-general}. In particular we illustrate
some connections with Example \ref{ex:Dickman-Poisson} (Section \ref{particular case }), and we present a specific case 
in which the random variables $\{R_n:n\geq 1\}$ are \emph{Neyman Type A} distributed random variables (see Example 
\ref{Poisson-Poisson} in Section \ref{poissonpoisson}).

\subsection{A particular case of Example \ref{ex:Dickman-Poisson}}\label{particular case }
In this section we explain how a particular case of Example \ref{ex:Dickman-Poisson} can be recovered also as a particular case 
of Example \ref{ex:Dickman-more-general}. This can be done by considering Example \ref{ex:Dickman-Poisson} with
$$\lambda_n=1- \pi_n,$$
and by taking Example \ref{ex:Dickman-more-general} with 
\begin{equation}\label{eq:Dickman-more-general-conditions}
	f_n(x)=e^{x-1}\ \mbox{for all}\ n\geq 1 
\end{equation}
(in particular $\lambda_n\to 0$ as $n\to\infty$ because, as remarked before, $\pi_n\to 1$). Note that $f_n$ does not depend on $n$ 
(i.e. $\{f_n:n\geq 1\}$ is a constant sequence of functions) and, moreover, \eqref{eq:condition-for-limit-law} holds because,
by Cesaro Theorem, we have
$$\lim_{n\to\infty}\frac{1}{\phi(n)}\sum_{k=1}^n\phi(k)\lambda_k=
\lim_{n\to\infty}\frac{\phi(n)(1-\pi_n)}{\phi(n)-\phi(n-1)}
=\lim_{n\to\infty}\frac{1-\left(\frac{\phi(n-1)}{\phi(n)}\right)^\nu}{1-\frac{\phi(n-1)}{\phi(n)}}=\nu.$$
Let us check conditions (a), (b) and (c) of Example \ref{ex:Dickman-more-general}. For condition (a)
it is easy to check \eqref{eq:condition-function-f} (we omit the details), and \eqref{eq:pmf-condition-0-and-1} holds 
because the definitions of $f_n$ and $\pi_n$ in \eqref{eq:Dickman-more-general-conditions} yield $q_n^{(0)}=e^{-\lambda_n}$
and $q_n^{(1)}=\lambda_ne^{-\lambda_n}$. Condition (b) holds because we have
$\mathbb{E}[e^{\theta R_n}]=e^{\lambda_n(e^\theta-1)}<\infty$ for every $\theta>0$ (actually for every $\theta\in\mathbb{R}$).
For condition (c) some further work is needed, and we  prove  the following lemma. 

\begin{lemma}\label{lem:uniform-limit}
	For every $\alpha_0 >0$ we have
	$$\lim_{x \to 0}\frac{e^{x\alpha}-1-x\alpha  }{ \alpha^2(1-e^{-x})^2}= \frac{1}{2}   .$$
	uniformly in $0< \alpha \leq \alpha_0$.
\end{lemma}

\begin{proof}
	Noting that
	$$ \frac{e^{x\alpha}-1-x\alpha  }{ \alpha^2(1-e^{-x})^2}=  \frac{e^{x\alpha}-1-x\alpha  }{ (\alpha x)^2(\frac{1-e^{-x}}{x} )^2} $$
	and recalling the relation  $\lim_{x \to 0}\frac{1-e^{-x}}{x} =1$, it suffices to prove that, for every $\alpha_0 >0$, we have
	$$\lim_{x \to 0}\frac{e^{x\alpha}-1-x\alpha  }{(x\alpha)^2}= \frac{1}{2} $$
	uniformly in $0< \alpha \leq \alpha_0$. Indeed, if we consider the function (defined on $(0,\infty)$)
	$$u \mapsto g(u) := \frac{e^{u}-1-u  }{u^2},$$
	we can say that $g(u)\geq \frac{1}{2}$ and is increasing in a neighborhood of $u=0$ (since its derivative is
	$$g^\prime(u)= \frac{u e^u + u -2 e^u +2}{u^3},$$
	which converges to $\frac{1}{6}$ as $u\to 0$). Hence for $0< \alpha \leq \alpha_0$ and $x>0$ small enough we have
	$$0\leq g(x\alpha)-\frac{1}{2} \leq  g(x\alpha_0)-\frac{1}{2} ,$$
	whence 
	$$0 \leq \lim_{x \to 0}\frac{e^{x\alpha}-1-x\alpha -\frac{1}{2}(x\alpha)^2}{(x\alpha)^2}\leq  \lim_{x \to 0}\frac{e^{x\alpha_0}-1-x\alpha_0 -\frac{1}{2}(x\alpha_0)^2}{(x\alpha_0)^2}=0.$$
\end{proof}

\noindent
Now we are in a position to check condition (c) in Example \ref{ex:Dickman-more-general}. First, for every 
$\tau\in \mathbb{R}$  we have
	\begin{align*}
		0&\leq\mathbb{E}[e^{\tau R_k}]-(q_k^{(0)}+e^\tau q_k^{(1)})=e^{\lambda_k(e^\tau-1)}-e^{-\lambda_k}-\lambda_ke^{-\lambda_k}e^\tau\\
		&=e^{-\lambda_k}(e^{\lambda_ke^\tau}-1-\lambda_ke^\tau)\leq e^{\lambda_ke^\tau}-1-\lambda_ke^\tau.
	\end{align*} 
Now we take $\theta \geq  |\tau| $ and use Lemma \ref{lem:uniform-limit} with $\alpha = e^\tau,  \alpha_0=e^\theta,  
\lambda_k =x$; then, for fixed $\epsilon >0$, there exists an integer $k_\theta$
such that, for $k>k_\theta$ (in what follows $C_\theta=\Big(\frac{1}{2}+ \epsilon\Big)e^{2\theta}$)
$$0 \leq e^{\lambda_ke^\tau}-1-\lambda_ke^\tau  \leq \Big(\frac{1}{2}+ \epsilon\Big)e^{2\tau}\big(1- e^{-\lambda_k}\big)^2
\leq\Big(\frac{1}{2}+ \epsilon\Big)e^{2\theta}\big(1- e^{-\lambda_k}\big)^2=C_\theta(1-q_k^{(0)})^2. $$

\subsection{A case of Example \ref{ex:Dickman-more-general} with Neyman Type A distributed random variables} \label{poissonpoisson}
In this section we present another specific case of Example \ref{ex:Dickman-more-general} in which $f_n$ depends on $n$
(i.e. $\{f_n:n\geq 1\}$ is not a constant sequence of functions, as happens for the functions defined in
\eqref{eq:Dickman-more-general-conditions}).

\begin{example}\label{Poisson-Poisson}
	We consider the situation of Section \ref{gen-framework}, and we assume that $\{R_n:n\geq 1\}$ is a
	sequence of (independent) \emph{Neyman Type A} distributed random variables, i.e. compound Poisson
	distributed random variables with Poisson distributed (i.i.d.) summands. More precisely we consider
	$$R_n=\sum_{k=1}^{N_n} X^{(n)}_k,$$
    where, for each fixed $n\geq 1$, the random variables $\{X^{(n)}_k:k\geq 1\}$ are i.i.d. and Poisson 
    distributed with mean $\beta_n$, and independent of another Poisson distributed $N_n$ with mean 
    $\lambda_n=\lambda$  (i.e. $\lambda_n$ does not depend
    on $n$). Then we have
    $$\mathbb{E}[e^{\theta R_n}]=\exp(\lambda(e^{\beta_n(e^\theta-1)}-1));$$
    moreover 
    $$q_n^{(0)}=e^{-\lambda (1-e^{-\beta_n})}\qquad\mbox{and}\qquad 
    q_n^{(1)}= e^{-\lambda (1-e^{-\beta_n})}\lambda \beta_n e^{-\beta_n}.$$
    Finally we also assume that $\beta_n\to 0$ as $n\to\infty$.
\end{example}

\noindent
In what follows we show that this is a particular case of Example \ref{ex:Dickman-more-general}
(see Proposition \ref{prop:Poisson-Poisson} below). We 
start noting that (even without requiring that $\beta_n\to 0$) 
\begin{align*}&
\mathbb{E}[e^{\theta R_n}]=\exp(\lambda(e^{\beta_n(e^\theta-1)}-1))= e^{-\lambda (1-e^{-\beta_n})}e^{\lambda e^{-\beta_n}(e^{\beta_n e^\theta}  -1)};
\end{align*}
so   for $0<\tau\leq\theta$ we have
\begin{align*}
	&\mathbb{E}[e^{\tau R_n}]-(q_n^{(0)}+e^\tau q_n^{(1)})
	=e^{-\lambda (1-e^{-\beta_n})}\Big\{e^{\lambda e^{-\beta_n}(e^{\beta_n e^\tau} -1)}-1-\lambda \beta_n e^{-\beta_n}e^\tau\Big\}\\
	&\leq e^{\lambda e^{-\beta_n}(e^{\beta_n e^\tau}  -1)}-1-\lambda \beta_n e^{-\beta_n} e^\tau.
\end{align*}
Hence \eqref{eq:UB-quadratic} holds thanks to the following lemma.

\begin{lemma} Assume that $\beta_n \to 0$. Then
  \begin{equation}\label{limit}  
    \lim_{n \to \infty  }\frac{e^{\lambda e^{-\beta_n}(e^{\beta_n e^\tau}  -1)}-1-\lambda e^{-\beta_n} \beta_n e^\tau}
    {(1-e^{-\lambda (1-e^{-\beta_n})} )^2}=\frac{\lambda+1}{2\lambda}e^{2\tau}.
  \end{equation}
  uniformly in $\tau\in(0,\theta]$.
\end{lemma}   

\begin{proof}  
 Put $x_n=\beta_n e^\tau$ and observe that, as $n \to \infty$,
 $$1-e^{-\lambda (1-e^{-\beta_n})}\sim \lambda (1-e^{-\beta_n})\sim \lambda \beta_n 
 = \lambda e^{-\tau}(\beta_n e^\tau)    =  \lambda e^{-\tau} x_n.$$
 So it suffices to prove the simplified relation
 $$\lim_{ n \to \infty}\frac{e^{\lambda e^{-\beta_n}(e^{x_n}  -1)}-1-\lambda e^{-\beta_n} x_n}{ x_n^2}
 =\frac{\lambda(\lambda+1)}{2}.$$
 Since $e^{-\beta_n}\to 1$, with a standard argument we  eliminate it, and we shall prove that
 \begin{equation}\label{1}
 \lim_{x \to 0 }\frac{e^{\lambda  (e^{x}  -1)}-1-\lambda   x}{ x^2}= \frac{\lambda(\lambda+1)}{2},
 \end{equation}
 or equivalently that
 $$\lim_{x \to 0 }\frac{e^{\lambda  (e^{x}  -1)}-1-\lambda   x -\frac{\lambda^2  }{2} x^2    }{ x^2}=\frac{\lambda}{2}.$$ 
 Since
 $$1+\lambda   x +\frac{\lambda^2  }{2} x^2 = e^{\lambda x}   +o(x^2)$$
 we prove equivalently that
 $$\lim_{x \to 0 }\frac{e^{\lambda  (e^{x}  -1)}-  e^{\lambda x}  }{ x^2}=\frac{\lambda    }{2}, $$
 and this holds  since 
 $$\lim_{x \to 0 }e^{\lambda x}\frac{e^{\lambda  (e^{x}  -1-x)}-  1  }{ x^2}= \lim_{x \to 0 }\frac{e^{\lambda  (e^{x}  -1-x)}-  1  }{ x^2}=\lambda\lim_{x \to 0 }\frac{     e^{x}  -1-x   }{ x^2} =\frac{\lambda    }{2}.$$ 
 Now we have to prove that the relation  \eqref{limit} is uniform in $0<\tau < \theta$. 
 First, the denominator in \eqref{limit} does not depend on $\tau$, so we can prove equivalently that
 $$ \lim_{n \to \infty  }\frac{e^{\lambda e^{-\beta_n}(e^{\beta_n e^\tau}  -1)}-1-\lambda e^{-\beta_n} \beta_n e^\tau}{ \lambda^2 \beta_n^2}=\frac{\lambda+1}{2\lambda}e^{2\tau},$$
 uniformly in $\tau\in(0,\theta]$, or 
 $$ \lim_{n \to \infty  }\frac{e^{\lambda e^{-\beta_n}(e^{\beta_n e^\tau}  -1)}-1-\lambda e^{-\beta_n} \beta_n e^\tau}{  \beta_n^2e^{2\tau}}=\frac{\lambda(\lambda+1)}{2 }.$$
 We eliminate $e^{-\beta_n} $ (arguing as we have done above); moreover we observe that, by \eqref{1}, for every 
 $\epsilon$, there exists $\delta$ such that, if $0<x_n < \delta$, then
 $$\Big|\frac{e^{\lambda  (e^{x_n }  -1)}-1-\lambda  x_n }{ x_n ^2}-\frac{\lambda(\lambda+1)}{2}\Big|< \epsilon.$$
 Then, since $ x_n =\beta_n e^\tau$, we need to take $\beta_n\in(0,\delta e^{-\tau}$); moreover, observing that 
 $e^{-\theta}< e^{-\tau}$, we can take $\beta_n\in(0,\delta  e^{-\theta})$ and we get the desired estimate for every 
 $\tau\in(0,\theta].$   
\end{proof}

\noindent
Now we are ready to show that Example \ref{Poisson-Poisson} is a specific case of Example \ref{ex:Dickman-more-general}.
We shall do that with a suitable choice of $\lambda$ and $\{\beta_n:n\geq 1\}$ in terms of a sequence $\{\pi_n:n\geq 1\}$ (which 
has the properties stated in Example \ref{ex:Dickman-more-general}, and in particular can be expressed in terms of the 
function $\phi$ in the Setting in Section \ref{gen-framework}; see \eqref{eq:def-pi(k)}). The functions $\{f_n:n\geq 1\}$ in Example 
\ref{ex:Dickman-more-general} will depend on the choice of $\lambda$ and $\{\beta_n:n\geq 1\}$.

\begin{proposition}\label{prop:Poisson-Poisson}
  In Example \ref{Poisson-Poisson} take $\lambda$ and $\beta_n$ such that
  $$1- \lambda \beta_n e^{-\beta_n}=\pi_n$$
  (where $\pi_n=\left(\frac{\phi(n-1)}{\phi(n)}\right)^\nu$ as in \eqref{eq:def-pi(k)}), and
  $$f_n(x) = \exp\Big({(x-1)\frac{e^{\beta_n}-1}{\beta_n}}\Big).$$
  Then we have
  $$q_n^{(0)}=f_n(\pi_n); \qquad q_n^{(1)} = f_n(\pi_n)(1-\pi_n).$$
\end{proposition}

\begin{proof}
	By the formulas in Example \ref{Poisson-Poisson} we get immediately
	$$q_n^{(1)}=q_n^{(0)}\lambda \beta_n e^{-\beta_n}= q_n^{(0)}(1-\pi_n);   $$
	 (note that $\beta_n\to 0$  because $\pi_n\to 1$ as $n\to\infty$). Finally we note that
	
	$$f_n(\pi_n)=\exp\Big({(\pi_n-1)\frac{e^{\beta_n}-1}{\beta_n}}\Big)
	=\exp\Big({-\lambda \beta_n e^{-\beta_n}\frac{e^{\beta_n}-1}{\beta_n}}\Big)
	=e^{-\lambda (1-e^{-\beta_n})}=q_n^{(0)}.$$
\end{proof}

\section{Some technical lemmas and proof of Theorem \ref{main}}\label{sec:final}

\subsection{Some technical lemmas} \label{lemmas}
First,  we recall   the well known  Abel's summation formula (see for instance \cite{Tenenbaum}):
\begin{equation}\label{eq:Abel-formula}
	\sum_{j=m}^na_jb_j=\left(\sum_{j=0}^na_j\right)b_n-\left(\sum_{j=0}^ma_j\right)b_m
	+\sum_{j=m}^{n-1}\left(\sum_{k=0}^ja_k\right)(b_j-b_{j+1})\quad \mbox{for}\ 0\leq m\leq n.
\end{equation}
The following lemma is proved in \cite{GiulianoMacciESAIM-PS}.
Set
$$s_{j,n}:=\sum_{k=j}^n\frac{\phi(k)-\phi(k-1)}{\phi^2(k)}.$$

\begin{lemma}[Lemma 3.1 in \cite{GiulianoMacciESAIM-PS}]\label{lem:dickman-LDP}
	Let $\phi:\mathbb{N}\to [0,\infty)$ be the strictly increasing
	function in \eqref{eq:conditions-on-phi(n)} and let $L(n)$ be
	defined by \eqref{eq:def-L(n)}.\\
	(i) Assume that \eqref{eq:concave-phi} holds. Then
	$s_{1,\infty}<\infty$; for $n>j\geq 1$, we have
	\begin{equation}\nonumber\label{eq:dickman-monotonicity}
		\phi(j)s_{j,n}-\phi(j+1)s_{j+1,n}\geq 0;
	\end{equation}
	for $n\geq j\geq 2$, we have
	\begin{equation}\label{eq:dickman-inequalities-chain}
		1-\frac{\phi(j)}{\phi(n+1)}\leq
		\phi(j)s_{j,n}\leq\frac{\phi(j)}{\phi(j-1)}\leq c,\ \mbox{for
			some}\ c\geq\frac{\phi(n)}{\phi(n-1)}\ \mbox{for all}\ n\geq 2.
	\end{equation}
	(ii) Assume that \eqref{eq:condition-for-limit-law} and
	$\lim_{n\to\infty}L(n)=\infty$ hold. Then
	\begin{equation}\nonumber\label{eq:dickman-limit-lambda-means}
		\lim_{n\to\infty}\frac{1}{L(n)}\sum_{k=1}^n\lambda_k=\nu.
	\end{equation}

\end{lemma}

\noindent
The following lemma  collects  other useful (but tedious) results.

\begin{lemma}\label{lem:estimates}
	Let us consider the Setting in Section \ref{gen-framework} (so, in particular, see \eqref {eq:def-L(n)} 
	for the definition of $L(n)$), and let $\nu$ be as in Example \ref{ex:Dickman-more-general}. Then
	\begin{align}&
		\lim_{n\to\infty}\frac{\log\phi(n)}{L(n)}=1;\label{eq:estimate1}\\ &\lim_{n\to\infty}\frac{1}{L(n)}\sum_{h=0}^n(1-\pi_h)=\nu\quad
	    \mbox{and}\quad\lim_{n\to\infty}\frac{1}{\log\phi(n)}\sum_{h=0}^n(1-\pi_h)=\nu;\label{eq:estimate2}\\ &\lim_{n\to\infty}\frac{1-f_n(\pi_n)}{L(n)-L(n-1)}=\nu;\label{eq:estimate3}\\ &\label{eq:estimate4}
    	\lim_{n\to\infty}\frac{\sum_{j=n_0+1}^n\log q_j^{(0)}}{L(n)}=-\nu\quad\mbox{for every integer}\ n_0;\\ &\label{eq:estimate5}
	    \lim_{n\to\infty}\frac{\sum_{j=n_0+1}^n(1-q_j^{(0)})^2}{L(n)}=0\quad\mbox{for every integer}\ n_0;\\ & 
	    \label{eq:LB-for-GE-e2-2-equivalent}
         \lim_{n\to\infty}\frac{1}{L(n)}\sum_{j=n_0+1}^n(1-\pi_j)^2=0;\\ &\label{eq:estimate-separate}   
	    \mbox{there exists}\ \beta>0\ \mbox{such that}\quad 0\leq e^{-x}+x-1\leq   x^2\quad\mbox{for every}\ x\ \in  (-  \beta, \beta).
	\end{align}
\end{lemma}
\begin{proof}
	We prove each statement separately.\\
	
	\bigskip
	\noindent
	Proof of \eqref{eq:estimate1}: by Cesaro Theorem, \eqref{eq:conditions-on-phi(n)} and \eqref{eq:def-L(n)} 
	(we omit some easy details) we have
	$$\lim_{n\to\infty}\frac{\log\phi(n)}{L(n)}
	=\lim_{n\to\infty}\frac{\log\phi(n)-\log\phi(n-1)}{L(n)-L(n-1)}
	=\lim_{n\to\infty}\frac{\log(\phi(n)/\phi(n-1))}{1-\frac{\phi(n-1)}{\phi(n)}}
	=1.$$
	
	\bigskip
	\noindent
	Proof of \eqref{eq:estimate2}: the first  limit holds  by Cesaro Theorem and \eqref{eq:def-pi(k)}, indeed we
	have
	$$\lim_{n\to\infty}\frac{1}{L(n)}\sum_{h=0}^n(1-\pi_h)
    =\lim_{n\to\infty}\frac{1-\pi_n}{L(n)-L(n-1)}
    =\lim_{n\to\infty}\frac{1-\left(\frac{\phi(n-1)}{\phi(n)}\right)^\nu}{1-\frac{\phi(n-1)}{\phi(n)}}=\nu;$$	
	the second limit in \eqref{eq:estimate2} follows by combining the first limit and \eqref{eq:estimate1}.\\
	
	\bigskip
	\noindent
    Proof of \eqref{eq:estimate3}: by the second limit in \eqref{limiti} (note that 
    $\pi_n\to 1$ as $n\to\infty$) and the above computations for the first limit in \eqref{eq:estimate2}
    we have
    $$\lim_{n\to\infty}\frac{1-f_n(\pi_n)}{L(n)-L(n-1)}
    =\lim_{n\to\infty}\frac{1-f_n(\pi_n)}{1-\pi_n}\cdot\frac{1-\pi_n}{L(n)-L(n-1)}=\nu.$$

    \bigskip
	\noindent
    Proofs of \eqref{eq:estimate4} and \eqref{eq:estimate5}:  by Cesaro Theorem, by taking into account
    $q_n^{(0)}=f_n(\pi_n)\to 1$ as $n\to\infty$ (see the first limit in \eqref{limiti}), and by 
    \eqref{eq:estimate3} (for the last equality), we have 
    \begin{align*}&
    	\lim_{n\to\infty}\frac{\sum_{j=n_0+1}^n\log q_j^{(0)}}{L(n)}
    	=\lim_{n\to\infty}\frac{\log q_n^{(0)}}{L(n)-L(n-1)} 
    	=\lim_{n\to\infty}\frac{q_n^{(0)}-1}{L(n)-L(n-1)}
    \\&	  =-\lim_{n\to\infty}\frac{1-f_n(\pi_n)}{L(n)-L(n-1)}=-\nu.
     \end{align*}
    Moreover, by reasoning in a similar way, we  have also 
    $$\lim_{n\to\infty}\frac{\sum_{j=n_0+1}^n(1-q_j^{(0)})^2}{L(n)}=\lim_{n\to\infty}\frac{(1-q_n^{(0)})^2}{L(n)-L(n-1)}
    =\lim_{n\to\infty}(1-f_n(\pi_n))\cdot\frac{1-f_n(\pi_n)}{L(n)-L(n-1)}=
    0\cdot\nu=0.$$ 
    
    \bigskip
	\noindent
    Proof of \eqref{eq:LB-for-GE-e2-2-equivalent}. We take $n_0$ such that, for $j>n_0$,
    $$\frac{1-f_j(\pi_j)}{1-\pi_j}\geq\frac{1}{2}$$
    (because $\pi_j\to 1$ as $j\to\infty$, and $\frac{1-f_j( x)}{1-x}\to 1$ as $x\to 1$, uniformly in $j$; 
    see \eqref{eq:condition-function-f} and \eqref{limiti}); then
    $$0\leq\frac{1}{4L(n)}\sum_{j=n_0+1}^n(1-\pi_j)^2\leq\frac{1}{L(n)}\sum_{j=n_0+1}^n(1-f_j(\pi_j))^2,$$
    and \eqref{eq:LB-for-GE-e2-2-equivalent} holds noting that, by \eqref{eq:estimate5} (we recall that 
    $f_j(\pi_j)=q_j^{(0)}$ by \eqref{eq:pmf-condition-0-and-1} in Example \ref{ex:Dickman-more-general}),
    $$\lim_{n\to\infty}\frac{1}{L(n)}\sum_{j=n_0+1}^n(1-f_j(\pi_j))^2=0.$$
    The proof of \eqref{eq:estimate-separate} is  standard and is omitted.
\end{proof}

\subsection{Proof of Theorem \ref{main}} \label{proof-of-the LDP}
	We prove the Proposition by applying the G\"artner--Ellis Theorem (as happens for the proof of 
	Proposition 3.2 in \cite{GiulianoMacciESAIM-PS}). So we have to prove that	
	$$\Lambda(\theta):=\lim_{n\to\infty}\frac{1}{L(n)}\log\mathbb{E}
	\left[e^{\theta\sum_{k=1}^n\frac{\phi(k)-\phi(k-1)}{\phi^2(k)}\sum_{j=1}^k\phi(j)R_j}\right]
	=\nu(e^\theta-1)$$ for all $\theta\in\mathbb{R}$. Indeed, since $\theta\mapsto\nu(e^\theta-1)$
	is a finite-valued and differentiable function (defined on $\mathbb{R}$), the LDP holds by the 
	G\"artner--Ellis Theorem with good rate function $\Lambda^*$ defined by
	$$\Lambda^*(x):=\sup_{\theta\in\mathbb{R}}\left\{\theta x-\nu(e^\theta-1)\right\},$$
	and $\Lambda^*$ coincides with the rate function $I$ in the statement of the 
	proposition (this can be easily checked with some easy computations). Moreover it 
	is useful to handle the above expression by following the same lines of the beginning of the
	proof of Proposition 3.2 in \cite{GiulianoMacciESAIM-PS}. By the independence of the random 
	variables $\{R_n:n\geq 1\}$ we have
	$$\log\mathbb{E}\left[e^{\theta\sum_{k=1}^n\frac{\phi(k)-\phi(k-1)}{\phi^2(k)}\sum_{j=1}^k\phi(j)R_j}\right]
	=\log\mathbb{E}\left[e^{\theta\sum_{j=1}^n\phi(j)s_{j,n}R_j}\right]
	=\sum_{j=1}^n\log\mathbb{E}\left[e^{\theta\phi(j)s_{j,n}R_j}\right];$$
    therefore we have to prove that
	$$\lim_{n\to\infty}\frac{1}{L(n)}\sum_{j=1}^n\log\mathbb{E}\left[e^{\theta\phi(j)s_{j,n}R_j}\right]=\nu(e^\theta-1)$$
	for all $\theta\in\mathbb{R}$.
	
	\noindent
	The case $\theta=0$ is immediate. Now we consider the case $\theta\not =0$. We start noting that, for every
	nonegative integer $n_0\geq 1$ (in what follows we shall take $n_0$ which depends on $\theta$ and an arbitrary
	small $\varepsilon>0$), we have
	$$\lim_{n\to\infty}\frac{1}{L(n)}\sum_{j=1}^{n_0}\log\mathbb{E}\left[e^{\theta\phi(j)s_{j,n}R_j}\right]=0;$$
	indeed 
	$$0\leq\frac{1}{L(n)}\sum_{j=1}^{n_0}\log\mathbb{E}\left[e^{\theta\phi(j)s_{j,n}R_j}\right]
	\leq\frac{1}{L(n)}\left(\log\mathbb{E}\left[e^{(0\vee\theta)\phi(1)s_{1,\infty}R_j}\right]
	+\sum_{j=2}^{n_0}\log\mathbb{E}\left[e^{(0\vee\theta) cR_j}\right]\right)$$

	\noindent
	where (see Lemma \ref{lem:dickman-LDP}) $s_{1,\infty}<\infty$ and $c$ is the constant in 
	\eqref{eq:dickman-inequalities-chain}. Thus it suffices to show that
	$$\lim_{n\to\infty}\frac{1}{L(n)}\sum_{j=n_0+1}^n\log\mathbb{E}\left[e^{\theta\phi(j)s_{j,n}R_j}\right]=\nu(e^\theta-1).$$
	Moreover, by \eqref{eq:dickman-inequalities-chain}, this limit relation holds if we have
	\begin{equation}\label{eq:LB-for-GE}
		\lim_{n\to\infty}\frac{1}{L(n)}\sum_{j=n_0+1}^n\log\mathbb{E}\left[e^{\theta(1-\frac{\phi(j)}{\phi(n+1)})R_j}\right]=\nu(e^\theta-1) 
	\end{equation}
    and
    \begin{equation}\label{eq:UB-for-GE}
    	\lim_{n\to\infty}\frac{1}{L(n)}\sum_{j=n_0+1}^n\log\mathbb{E}\left[e^{\theta\frac{\phi(j)}{\phi(j-1)}R_j}\right]=\nu(e^\theta-1).
    \end{equation}

\noindent
    We start with the proof of \eqref{eq:LB-for-GE}; the proof of \eqref{eq:UB-for-GE} is similar, 
    and we give some details below. We take into account $q_j^{(1)}=q_j^{(0)}(1-\pi_j)$ (see 
    \eqref{eq:pmf-condition-0-and-1} in Example \ref{ex:Dickman-more-general}), and we have
    \begin{multline*}
    	\log\mathbb{E}\left[e^{\theta(1-\frac{\phi(j)}{\phi(n+1)})R_j}\right]\\
    	=\log\mathbb{E}\left[e^{\theta(1-\frac{\phi(j)}{\phi(n+1)})R_j}\right]
    	-\log\left(q_j^{(0)}+e^{\theta(1-\frac{\phi(j)}{\phi(n+1)})}q_j^{(1)}\right)
    	+\log\left(q_j^{(0)}+e^{\theta(1-\frac{\phi(j)}{\phi(n+1)})}q_j^{(1)}\right)\\
    	=\log\mathbb{E}\left[e^{\theta(1-\frac{\phi(j)}{\phi(n+1)})R_j}\right]
    	-\log\left(q_j^{(0)}+e^{\theta(1-\frac{\phi(j)}{\phi(n+1)})}q_j^{(1)}\right)
    	+\log q_j^{(0)}+\log\left(1+e^{\theta(1-\frac{\phi(j)}{\phi(n+1)})}(1-\pi_j)\right);
    \end{multline*}
    then, since
    $$\lim_{n\to\infty}\frac{1}{L(n)}\sum_{j=n_0+1}^n\log q_j^{(0)} =-\nu$$
    by \eqref{eq:estimate4}, we get \eqref{eq:LB-for-GE} if we prove
    \begin{equation}\label{eq:LB-for-GE-e1}
    	\lim_{n\to\infty}\frac{1}{L(n)}\sum_{j=n_0+1}^n\left\{\log\mathbb{E}\left[e^{\theta(1-\frac{\phi(j)}{\phi(n+1)})R_j}\right]
    	-\log\left(q_j^{(0)}+e^{\theta(1-\frac{\phi(j)}{\phi(n+1)})}q_j^{(1)}\right)\right\}=0,
    \end{equation}
    and
    \begin{equation}\label{eq:LB-for-GE-e2}
    	\lim_{n\to\infty}\frac{1}{L(n)}\sum_{j=n_0+1}^n\log\left(1+e^{\theta(1-\frac{\phi(j)}{\phi(n+1)})}(1-\pi_j)\right)=\nu e^\theta.
    \end{equation}
    Now we prove \eqref{eq:LB-for-GE-e1}. By   Lagrange Theorem (note that
     $$q_j^{(0)}+e^{\theta(1-\frac{\phi(j)}{\phi(n+1)})}q_j^{(1)}<\mathbb{E}\left[e^{\theta(1-\frac{\phi(j)}{\phi(n+1)})R_j}\right]$$  
    because the summands are positive) there exists 
    $\xi_{n,j}\in\left(q_j^{(0)}+e^{\theta(1-\frac{\phi(j)}{\phi(n+1)})}q_j^{(1)},
    \mathbb{E}\left[e^{\theta(1-\frac{\phi(j)}{\phi(n+1)})R_j}\right]\right)$ such that
    \begin{multline*}
    	0\leq\log\mathbb{E}\left[e^{\theta(1-\frac{\phi(j)}{\phi(n+1)})R_j}\right]
    	-\log\left(q_j^{(0)}+e^{\theta(1-\frac{\phi(j)}{\phi(n+1)})}q_j^{(1)}\right)\\
    	=\frac{1}{\xi_{n,j}}\left\{\mathbb{E}\left[e^{\theta(1-\frac{\phi(j)}{\phi(n+1)})R_j}\right]
    	-\left(q_j^{(0)}+e^{\theta(1-\frac{\phi(j)}{\phi(n+1)})}q_j^{(1)}\right)\right\};
    \end{multline*}
    moreover, by \eqref{eq:UB-quadratic} in Example \ref{ex:Dickman-more-general} with 
    $\tau = \theta(1-\frac{\phi(j)}{\phi(n+1)})$ 
    we have  
    $$0\leq\mathbb{E}\left[e^{\theta(1-\frac{\phi(j)}{\phi(n+1)})R_j}\right]
    -\left(q_j^{(0)}+e^{\theta(1-\frac{\phi(j)}{\phi(n+1)})}q_j^{(1)}\right)\\
    \leq\ {C}_\theta(1-q_j^{(0)})^2,$$
    for every $j > k_\theta$.
    Now we take $n_0$ such that, for $j>n_0$, $f_j(\pi_j)\geq\frac{1}{2}$ (see \eqref{limiti}); then (again for $j\geq n_0$)
    $$\xi_{n,j}\geq q_j^{(0)}+e^{\theta(1-\frac{\phi(j)}{\phi(n+1)})}q_j^{(1)}\geq q_j^{(0)}=f(\pi_j)\geq\frac{1}{2},$$
    and therefore
    $$0\leq\log\mathbb{E}\left[e^{\theta(1-\frac{\phi(j)}{\phi(n+1)})R_j}\right]
    -\log\left(q_j^{(0)}+e^{\theta(1-\frac{\phi(j)}{\phi(n+1)})}q_j^{(1)}\right)\leq 2 C_\theta(1-q_j^{(0)})^2.$$
    In conclusion we have
    $$0\leq\frac{1}{L(n)}\sum_{j=n_0+1}^n\left\{\log\mathbb{E}\left[e^{\theta(1-\frac{\phi(j)}{\phi(n+1)})R_j}\right]
    -\log\left(q_j^{(0)}+e^{\theta(1-\frac{\phi(j)}{\phi(n+1)})}q_j^{(1)}\right)\right\}
    \leq\frac{2 {C}_\theta}{L(n)}\sum_{j=n_0+1}^n(1-q_j^{(0)})^2,$$
    and we get \eqref{eq:LB-for-GE-e1} by \eqref{eq:estimate5}.
    
    \bigskip
    \noindent
    Now we prove \eqref{eq:LB-for-GE-e2}. It is well-known that there exists $m>0$ such that
    \begin{equation}\label{eq:linearization-logarithm}
    	|\log(1+x)-x|\leq mx^2\quad\mbox{for}\ |x|<\frac{1}{2}.
    \end{equation}
    Moreover
    $$0\leq\max_{2\leq j\leq n}e^{\theta(1-\frac{\phi(j)}{\phi(n+1)})}(1-\pi_n)
    \leq\max_{2\leq j\leq n}(1\vee e^\theta)(1-\pi_n)\to 0\quad\mbox{as}\ n\to\infty$$
    (here we take into account that $1-\frac{\phi(j)}{\phi(n+1)}\in[0,1]$). Thus we take $n_0$ such that
    $(1\vee e^\theta)(1-\pi_n)<\frac{1}{2}$ for $n>n_0$ (because $\pi_n\to 1$ as $n\to\infty$ by \eqref{limiti}), and we have
    $$e^{\theta(1-\frac{\phi(j)}{\phi(n+1)})}(1-\pi_n)<\frac{1}{2}\quad\mbox{for}\ n\geq j>n_0.$$
    Then, by using \eqref{eq:linearization-logarithm} with $x=e^{\theta(1-\frac{\phi(j)}{\phi(n+1)})}(1-\pi_n)$,
    we get
    \begin{align*}&
    	\left|\frac{1}{L(n)}\sum_{j=n_0+1}^n\log\left(1+e^{\theta(1-\frac{\phi(j)}{\phi(n+1)})}(1-\pi_j)\right)
    	-\frac{1}{L(n)}\sum_{j=n_0+1}^ne^{\theta(1-\frac{\phi(j)}{\phi(n+1)})}(1-\pi_j)\right|\\&
    	\leq\frac{m}{L(n)}\sum_{j=n_0+1}^ne^{2\theta(1-\frac{\phi(j)}{\phi(n+1)})}(1-\pi_j)^2.
    \end{align*}
    Hence we obtain \eqref{eq:LB-for-GE-e2} if we prove that
    \begin{equation}\label{eq:LB-for-GE-e2-1}
    	\lim_{n\to\infty}\frac{1}{L(n)}\sum_{j=n_0+1}^ne^{\theta(1-\frac{\phi(j)}{\phi(n+1)})}(1-\pi_j)=\nu e^\theta
    \end{equation}
    and
    \begin{equation}\label{eq:LB-for-GE-e2-2}
    	\lim_{n\to\infty}\frac{1}{L(n)}\sum_{j=n_0+1}^ne^{2\theta(1-\frac{\phi(j)}{\phi(n+1)})}(1-\pi_j)^2=0.
    \end{equation}
    The second limit \eqref{eq:LB-for-GE-e2-2} can be easily proved noting that
    $$0\leq\frac{1}{L(n)}\sum_{j=n_0+1}^ne^{2\theta(1-\frac{\phi(j)}{\phi(n+1)})}(1-\pi_j)^2
    \leq\frac{(1\vee e^{2\theta})}{L(n)}\sum_{j=n_0+1}^n(1-\pi_j)^2;$$
    thus \eqref{eq:LB-for-GE-e2-2} follows from \eqref{eq:LB-for-GE-e2-2-equivalent}.

    \medskip
    \noindent
    The first limit \eqref{eq:LB-for-GE-e2-1} will be proved in several steps. We 
    apply  the Abel formula \eqref{eq:Abel-formula} with $m=n_0+1$, $a_1=\ldots=a_{n_0}=0$
    and
    $$a_j=1-\pi_j\quad\mbox{and}\quad b_j=e^{\theta(1-\frac{\phi(j)}{\phi(n+1)})}\quad\mbox{for}\ j=n_0+2,\ldots,n;$$
    then
    \begin{multline*}
    	\sum_{j=n_0+1}^n(1-\pi_j)e^{\theta(1-\frac{\phi(j)}{\phi(n+1)})}
    	=\left(\sum_{j=0}^n(1-\pi_j)\right)e^{\theta(1-\frac{\phi(n)}{\phi(n+1)})}
    	-\left(\sum_{j=0}^{n_0}(1-\pi_j)\right)e^{\theta(1-\frac{\phi(n_0)}{\phi(n+1)})}\\
    	+\sum_{j=n_0+1}^{n-1}\left(\sum_{k=0}^j(1-\pi_k)\right)\left(e^{\theta(1-\frac{\phi(j)}{\phi(n+1)})}
    	-e^{\theta(1-\frac{\phi(j+1)}{\phi(n+1)})}\right).
    \end{multline*}
    We remark that
    $$\lim_{n\to\infty}\frac{1}{L(n)}\left(\sum_{j=0}^n(1-\pi_j)\right)e^{\theta(1-\frac{\phi(n)}{\phi(n+1)})}=\nu$$
    by \eqref{eq:estimate2} and the second limit in \eqref{eq:conditions-on-phi(n)}, and
    $$\lim_{n\to\infty}\frac{1}{L(n)}\left(\sum_{j=0}^{n_0}(1-\pi_j)\right)e^{\theta(1-\frac{\phi(n_0)}{\phi(n+1)})}=0$$
    because
    $$0\leq\frac{1}{L(n)}\left(\sum_{j=0}^{n_0}(1-\pi_j)\right)e^{\theta(1-\frac{\phi(n_0)}{\phi(n+1)})}
    \leq\frac{(1 \vee e^\theta)}{L(n)}\left(\sum_{j=0}^{n_0}(1-\pi_j)\right)\to 0\quad\mbox{as}\ n\to\infty$$
    (in the last inequality we take into account that $1-\frac{\phi(n_0)}{\phi(n+1)}\in[0,1]$). Thus we prove
    \eqref{eq:LB-for-GE-e2-1} if we show that
    \begin{equation}\label{eq:limit-to-prove-after-Abel}
    	 \lim_{n\to\infty}\frac{1}{L(n)}\sum_{j=n_0+1}^{n-1}\left(\sum_{k=0}^j(1-\pi_k)\right)\left(e^{\theta(1-\frac{\phi(j)}{\phi(n+1)})}
    	-e^{\theta(1-\frac{\phi(j+1)}{\phi(n+1)})}\right)=\nu(e^\theta-1).
    \end{equation}
    By \eqref{eq:estimate1} and the second limit in \eqref{eq:estimate2} we can say that, for every arbitrarily small $\varepsilon>0$, 
    there exists $n_0$ such that for $n>n_0$ we have
    $$(\nu-\varepsilon)\log\phi(n)\leq\sum_{k=0}^n(1-\pi_k)\leq(\nu+\varepsilon)\log\phi(n)$$
    and
    $$(1-\varepsilon)\log\phi(n)\leq L(n)\leq(1+\varepsilon)\log\phi(n);$$
    therefore, noting that we can take $n_0$ large enough in order to have $\log\phi(j)>0$ for
    $j>n_0$ (because $\phi$ is strictly increasing and $\phi(n)\to\infty$ as $n\to\infty$), we have
    \begin{align*}&
    	\frac{\nu-\varepsilon}{1+\varepsilon}\cdot\frac{1}{\log\phi(n)}
    	\sum_{j=n_0+1}^{n-1}\log\phi(j)\left(e^{\theta(1-\frac{\phi(j)}{\phi(n+1)})}
    	-e^{\theta(1-\frac{\phi(j+1)}{\phi(n+1)})}\right)\\&
    	\leq\frac{1}{L(n)}\sum_{j=n_0+1}^{n-1}\left(\sum_{k=0}^j(1-\pi_k)\right)\left(e^{\theta(1-\frac{\phi(j)}{\phi(n+1)})}
    	-e^{\theta(1-\frac{\phi(j+1)}{\phi(n+1)})}\right)\\&
    	\leq\frac{\nu+\varepsilon}{1-\varepsilon}\cdot\frac{1}{\log\phi(n)}
    	\sum_{j=n_0+1}^{n-1}\log\phi(j)\left(e^{\theta(1-\frac{\phi(j)}{\phi(n+1)})}
    	-e^{\theta(1-\frac{\phi(j+1)}{\phi(n+1)})}\right).
    \end{align*}
    Then, by the arbitrariness of $\varepsilon$, we prove \eqref{eq:limit-to-prove-after-Abel} if we show that
    $$\lim_{n\to\infty}\frac{1}{\log\phi(n)}
    \sum_{j=n_0+1}^{n-1}\log\phi(j)\left(e^{\theta(1-\frac{\phi(j)}{\phi(n+1)})}
    -e^{\theta(1-\frac{\phi(j+1)}{\phi(n+1)})}\right)=e^\theta-1.$$
    In fact, since
    $$e^{\theta(1-\frac{\phi(j)}{\phi(n+1)})}-e^{\theta(1-\frac{\phi(j+1)}{\phi(n+1)})}
    =e^\theta\cdot e^{-\theta\frac{\phi(j)}{\phi(n+1)}}\left(1-e^{-\theta\frac{\phi(j+1)-\phi(j)}{\phi(n+1)}}\right),$$
     \eqref{eq:limit-to-prove-after-Abel} will be proved if we show that
    $$\lim_{n\to\infty}\frac{1}{\log\phi(n)}
    \sum_{j=n_0+1}^{n-1}\log\phi(j)\cdot e^{-\theta\frac{\phi(j)}{\phi(n+1)}}
    \left(1-e^{-\theta\frac{\phi(j+1)-\phi(j)}{\phi(n+1)}}\right)=1-e^{-\theta};$$
    moreover, by \eqref{eq:estimate-separate}, for $n$ large enough so that $\frac{\phi(j+1)-\phi(j)}{\phi(n+1)}< 
    \frac{\beta }{|\theta|}$  
    (recall that,  by \eqref{eq:conditions-on-phi(n)} and \eqref{eq:concave-phi}, $\phi(j+1)-\phi(j)$ is 
    non-increasing and positive, hence bounded and, by \eqref{eq:conditions-on-phi(n)}, $\phi(n+1)$ diverges) 
    we have
    \begin{align*}&
    	\left|\frac{1}{\log\phi(n)}
    	\sum_{j=n_0+1}^{n-1}\log\phi(j)\cdot e^{-\theta\frac{\phi(j)}{\phi(n+1)}}
    	\left(1-e^{-\theta\frac{\phi(j+1)-\phi(j)}{\phi(n+1)}}-\theta\frac{\phi(j+1)-\phi(j)}{\phi(n+1)}\right)
    	\right|\\&
    	\leq\frac{ \theta^2}{\log\phi(n)}
    	\sum_{j=n_0+1}^{n-1}\log\phi(j)\cdot e^{-\theta\frac{\phi(j)}{\phi(n+1)}}
    	\left(\frac{\phi(j+1)-\phi(j)}{\phi(n+1)}\right)^2,
    \end{align*}
    and therefore we prove \eqref{eq:limit-to-prove-after-Abel} if we show that
    \begin{equation}\label{eq:limit-to-prove-after-Abel-1}
    	\lim_{n\to\infty}\frac{1}{\log\phi(n)}
    	\sum_{j=n_0+1}^{n-1}\log\phi(j)\cdot e^{-\theta\frac{\phi(j)}{\phi(n+1)}}\cdot
    	\frac{\phi(j+1)-\phi(j)}{\phi(n+1)}=\frac{1-e^{-\theta}}{\theta}
    \end{equation}
    and
    \begin{equation}\label{eq:limit-to-prove-after-Abel-2}
   	\lim_{n\to\infty}\frac{1}{\log\phi(n)}
   	\sum_{j=n_0+1}^{n-1}\log\phi(j)\cdot e^{-\theta\frac{\phi(j)}{\phi(n+1)}}
   	\left(\frac{\phi(j+1)-\phi(j)}{\phi(n+1)}\right)^2=0.
   \end{equation}
   Concerning \eqref{eq:limit-to-prove-after-Abel-1}, note that
   \begin{align*}&
   	\frac{1}{\log\phi(n)}
   	\sum_{j=n_0+1}^{n-1}\log\phi(j)\cdot e^{-\theta\frac{\phi(j)}{\phi(n+1)}}\cdot
   	\frac{\phi(j+1)-\phi(j)}{\phi(n+1)}\\&
   	=\frac{1}{\log\phi(n)}
   	\sum_{j=n_0+1}^{n-1}\log\frac{\phi(j)}{\phi(n+1)}\cdot e^{-\theta\frac{\phi(j)}{\phi(n+1)}}\cdot
   	\frac{\phi(j+1)-\phi(j)}{\phi(n+1)}\\&
   	+\frac{\log\phi(n+1)}{\log\phi(n)}
   	\sum_{j=n_0+1}^{n-1}e^{-\theta\frac{\phi(j)}{\phi(n+1)}}\cdot
   	\frac{\phi(j+1)-\phi(j)}{\phi(n+1)}
   \end{align*}
   and, by taking the limit as $n\to\infty$ (we have two integral sums), we get
   $$0\cdot\int_0^1\log x\cdot e^{-\theta x}dx+1\cdot\int_0^1e^{-\theta x}dx=\frac{1-e^{-\theta}}{\theta}.$$
   As far as \eqref{eq:limit-to-prove-after-Abel-2} is concerned, in a similar way we have
   \begin{align*}&
   	\frac{1}{\log\phi(n)}
   	\sum_{j=n_0+1}^{n-1}\log\phi(j)\cdot e^{-\theta\frac{\phi(j)}{\phi(n+1)}}
   	\left(\frac{\phi(j+1)-\phi(j)}{\phi(n+1)}\right)^2\\&
   	=\frac{1}{\log\phi(n)}
   	\sum_{j=n_0+1}^{n-1}\log\frac{\phi(j)}{\phi(n+1)}\cdot e^{-\theta\frac{\phi(j)}{\phi(n+1)}}
   	\left(\frac{\phi(j+1)-\phi(j)}{\phi(n+1)}\right)^2\\&
   	+\frac{\log\phi(n+1)}{\log\phi(n)}
   	\sum_{j=n_0+1}^{n-1}e^{-\theta\frac{\phi(j)}{\phi(n+1)}}
   	\left(\frac{\phi(j+1)-\phi(j)}{\phi(n+1)}\right)^2
   \end{align*}
   and, by \eqref{eq:concave-phi}, we get
    \begin{align*}&
   	0\leq\frac{1}{\log\phi(n)}
   	\sum_{j=n_0+1}^{n-1}\log\phi(j)\cdot e^{-\theta\frac{\phi(j)}{\phi(n+1)}}
   	\left(\frac{\phi(j+1)-\phi(j)}{\phi(n+1)}\right)^2\\&
   	\leq\frac{\phi(1)-\phi(0)}{\phi(n+1)}\left(\frac{1}{\log\phi(n)}
   	\sum_{j=n_0+1}^{n-1}\log\frac{\phi(j)}{\phi(n+1)}\cdot e^{-\theta\frac{\phi(j)}{\phi(n+1)}}
   	\cdot\frac{\phi(j+1)-\phi(j)}{\phi(n+1)}\right.\\&
   	\left.+\frac{\log\phi(n+1)}{\log\phi(n)}
   	\sum_{j=n_0+1}^{n-1}e^{-\theta\frac{\phi(j)}{\phi(n+1)}}
   	\cdot\frac{\phi(j+1)-\phi(j)}{\phi(n+1)}\right);
   \end{align*}
   thus, by taking the limit as $n\to\infty$ in the right hand side (and by taking into account the limit 
   computed above when we have checked \eqref{eq:limit-to-prove-after-Abel-1}), we get
   $$0\cdot\left(0\cdot\int_0^1\log x\cdot e^{-\theta x}dx+1\cdot\int_0^1e^{-\theta x}dx\right)=0$$
   and \eqref{eq:limit-to-prove-after-Abel-2} is proved. This completes the proof of \eqref{eq:LB-for-GE}
   (for $\theta>0$).
   
   \noindent
   The proof of \eqref{eq:UB-for-GE} is similar to the proof of \eqref{eq:LB-for-GE}.
   Indeed, for every arbitrarily small $\varepsilon>0$, there exists $n_0$ such that for $n>n_0$ we have
   $$1-\varepsilon\leq\frac{\phi(n)}{\phi(n-1)}\leq 1+\varepsilon,$$
   and therefore, if $\theta >0$
   $$\frac{1}{L(n)}\sum_{j=n_0+1}^n\log\mathbb{E}\left[e^{\theta(1-\varepsilon)R_j}\right]
   \leq\frac{1}{L(n)}\sum_{j=n_0+1}^n\log\mathbb{E}\left[e^{\theta\frac{\phi(j)}{\phi(j-1)}R_j}\right]
   \leq\frac{1}{L(n)}\sum_{j=n_0+1}^n\log\mathbb{E}\left[e^{\theta(1+\varepsilon)R_j}\right],$$
   (while for $\theta < 0$ the above inequalities are reversed).
   Moreover, by adapting the computations above to prove \eqref{eq:LB-for-GE}, we have
   $$\frac{1}{L(n)}\sum_{j=n_0+1}^n\log\mathbb{E}\left[e^{\theta(1\pm\varepsilon)R_j}\right]
   =\nu(e^{\theta\pm\varepsilon}-1)$$
   (just replace the parts with $1-\frac{\phi(j)}{\phi(+1)}$ with $1\pm\varepsilon$);
   thus \eqref{eq:UB-for-GE} holds  by the arbitrariness of $\varepsilon$.

\subsection*{Acknowledgements}
The authors thank an anonymous referee for some useful comments.

\subsection*{Funding}
The authors acknowledge the support of MUR Excellence Department Project 
awarded to the Department of Mathematics, University of Rome Tor Vergata (CUP E83C23000330006), of 
University of Rome Tor Vergata (CUP E83C25000630005) Research Project METRO, and of INdAM-GNAMPA.

\end{document}